\documentclass[11pt]{article}

\pdfoutput=1

\usepackage{amsmath}
\usepackage{amssymb}
\usepackage{amsthm}
\usepackage{authblk}
\usepackage{bm}
\usepackage{booktabs}
\usepackage[a4paper,bottom=3cm]{geometry}
\usepackage{graphicx}
\usepackage{multicol}
\usepackage[numbers,sort&compress]{natbib}
\usepackage{siunitx}
\usepackage{subcaption}

\newtheorem{definition}{Definition}[section]

\newtheorem{theorem}[definition]{Theorem}

\newcommand{\R}{\mathbb{R}}
\newcommand{\abs}[1]{\lvert#1\rvert}
\newcommand{\argmin}{\operatorname*{arg\,min}}
\newcommand{\grid}{\mathcal{G}}

\newcommand{\ind}{\chi}
\newcommand{\minmod}{\operatorname*{minmod}}
\newcommand{\sign}{\operatorname{sign}}

\title{Constrained Reconstruction in MUSCL-type Finite Volume Schemes}

\author{Christoph Gersbacher}
\author{Martin Nolte}

\affil{\small Department of Applied Mathematics, University of Freiburg, \\
  Hermann-Herder-Str. 10, D-79104 Freiburg, Germany.}

\date{}

\begin{document}

\maketitle

\begin{abstract}
In this paper we are concerned with the stabilization of MUSCL-type finite
volume schemes in arbitrary space dimensions. We consider a number of limited
reconstruction techniques which are defined in terms inequality-constrained
linear or quadratic programming problems on individual grid elements. No
restrictions to the conformity of the grid or the shape of its elements are
made. In the special case of Cartesian meshes a novel QP reconstruction is
shown to coincide with the widely used Minmod reconstruction. The accuracy
and overall efficiency of the stabilized second-order finite volume schemes
is supported by numerical experiments.
\end{abstract}

\section{Introduction}

First-order finite volume schemes are a class of discontinuous finite element
methods widely used in the numerical solution of first-order hyperbolic
conservation laws. They are particularly popular for their conservation
properties and their robustness. Finite volume schemes can be applied to
arbitrary shaped grid elements and locally adapted grids while still being easy
to implement. However, due to the large amount of dissipation built into the
first-order scheme, discontinuities in the exact solution may be heavily
smeared out.

In this paper we will be concerned with second-order MUSCL-type finite volume
schemes of formally second order instead. The MUSCL-approach was introduced by
\citet{vanLeer1979} and is based on the reconstruction of piecewise linear
functions from piecewise constant data. The second-order scheme in general
provides a much better resolution than the first-order finite volume method. It
is, however, prone to developing spurious oscillations and unphysical values
that may result in the immediate breakdown of a numerical simulation and
requires a suitable stabilization. A common approach to their stabilization
involves suitable slope limiters to prevent spurious oscillations. Only a few
stabilization techniques are applicable to general unstructured grids in $d$
space dimensions, see, e.g., \cite{Dedner2011}.

Here, we follow a more recent approach to the design of limited reconstruction
operators \cite{Buffard2010, May2013, Chen2016}. We consider a number of
reconstruction operators defined in terms of local inequality-constrained
linear or quadratic minimization problems. No restrictions to the conformity of
the grid or the shape of individual grid elements are imposed. We are able to
prove that in the special case of Cartesian meshes our QP reconstruction
coincides with the $d$-dimensional Minmod reconstruction, which illustrates the
reliability of the stabilized finite volume method. The general approach allows
for a number of modifications, and we briefly discuss a positivity-preserving
stabilization for the Euler equations of gas dynamics.

\section{Stabilization of MUSCL-type finite volume schemes}

In this section we want to briefly revisit MUSCL-type finite volume schemes on
arbitrary meshes. In the following, let $\Omega \subset \R^d$, $d > 0$, be a
bounded domain. We consider a first-order system
\begin{equation} \label{eq:conservation law}
  \begin{aligned}
    \partial_t u + \nabla \cdot F(u) & = 0 \quad \text{in $\Omega \times (0,T)$}, \\
    u(\cdot, 0) & = u_0 \quad \text{in $\Omega$}
  \end{aligned}
\end{equation}
subject to suitable boundary conditions. Here,
\begin{math}
  u: \Omega \times (0,T) \rightarrow \mathcal{U}
\end{math}
is an unknown function with values in the set of states $\mathcal{U} \subset
\R^r$. The function
\begin{math}
  u_0: \Omega \rightarrow \mathcal{U}
\end{math}
denotes some given initial data, and
\begin{math}
   F: \mathcal{U} \rightarrow \R^{d \times r}
\end{math}
is the so-called convective flux.

\paragraph{General notation} %
Now, let $\grid$ be a suitable partition of the computational domain into
closed convex polytopes with non-overlapping interior. For each element $E \in
\grid$ we denote by $\mathcal{N}(E)$ the set of its neighboring elements, i.e.,
\begin{equation*}
  \mathcal{N}(E)
    = \bigl\{E' \in \grid \setminus \{E \} \bigm\vert \dim(E \cap E') = d-1 \bigr\},
\end{equation*}
where for each boundary segment $E \cap \partial \Omega$, $E \in \grid$, we
assume the existence of an exterior, possibly degenerate ghost cell $E' \in
\mathcal{N}(E)$. By $X^k_\grid$ we denote the piecewise polynomial spaces of
order at most $k$ on $\grid$,
\begin{equation*}
  X^k_\grid
    = \bigl\{u \in L^\infty(\Omega, \R^r)
      \bigm\vert u_{\restriction{E}} \in (P^k(E))^r \text{ for all } E \in \grid
      \bigr\}.
\end{equation*}
Furthermore, for $u \in X^0_\grid$ we write
\begin{equation*}
  u(x) = \sum_{E \in \grid} u_E \ind_E(x),
\end{equation*}
where $\ind_E$ denotes the indicator function for the element $E \in \grid$ and
$u_E$ denotes the local cell average. The centroid of a convex polytope $E
\subset \R^d$ will be denoted by $x_E$.

\paragraph{MUSCL-type finite volume schemes} %
Next, we define the general second-order finite volume scheme. For each element
$E \in \grid$ and $E' \in \mathcal{N}(E)$ denote by
\begin{math}
  G_{E,E'}: C^\infty(E \cap E') \times C^\infty(E \cap E') \to \R^r
\end{math}
a conservative numerical flux from $E$ to $E'$ that consistent with $F$, i.e.,
\begin{equation*}
  G_{E,E'}(u, u) = \int_{E \cap E'} F(u)\,\cdot \nu_{E,E'}\,dx
  \quad \text{for all $u \in C^\infty(E \cap E')$},
\end{equation*}
where $\nu_{E,E'} \in \R^d$ denotes the unit outer normal to $E$ on the
intersection $E \cap E'$. After performing a spatial discretization of
\eqref{eq:conservation law}, we seek $u: [0,T] \rightarrow X^0_\grid$ such that
\begin{equation} \label{eqn:muscl_scheme}
\begin{aligned}
  \frac{\mathrm{d}}{\mathrm{d}t} \, u(t)
    & = - \sum_{E \in \grid} \frac{\ind_E}{\abs{E}} \sum_{E' \in \mathcal{N}(E)} G_{E,E'} \bigl(\mathcal{R} u(t)_{\restriction E}, \mathcal{R} u(t)_{\restriction E'}\bigr)
      \quad \text{for $t \in (0, T)$}, \\
  u(0) & = \Pi^0_\grid u_0 := \sum_{E \in \grid} \frac{\ind_E}{\abs{E}} \int_E u_0\,\mathrm{d}x.
\end{aligned}
\end{equation}
Here,
\begin{math}
  \mathcal{R} : X^0_\grid \to X^1_\grid
\end{math}
denotes a reconstruction operator mapping piecewise constant to piecewise
linear data. The operator is assumed to be locally mass-conservative, i.e.,
for all $u \in X^0_\grid$ it holds
\begin{equation*}
  \mathcal{R}u(x_E) = u_E \quad \text{for all $E \in \grid$}.
\end{equation*}
Note that in Equation \eqref{eqn:muscl_scheme} we made implicit use of
so-called ghost values, which must be determined from the given set of boundary
conditions. For the higher-order discretization in time we use a second-order
accurate Runge-Kutta method.

Obviously, the key ingredient to achieving second-order accuracy is the
reconstruction operator $\mathcal{R}$. The design of such operators is a
delicate matter which ultimately will affect the robustness and accuracy of the
overall numerical scheme. The generalization of techniques developed for the
one-dimensional case to multiple space dimensions is not always obvious, in
particular with respect to arbitrary shaped grid elements and possibly
non-conforming grids.

\paragraph{Limited least squares fitted polynomials} %

Arguably the most popular class of stabilized reconstruction techniques is due
to \citet{Barth1989}. It is based on a two-step procedure to be illustrated by
a \emph{limited least squares fit}.

For the sake of simplicity, we restrict the following presentation to scalar
functions. Let $u \in X^0_\grid$ be a piecewise constant function and $E \in
\grid$ a fixed grid cell. We fix non-negative weights $\omega_{E,E'}$, $E' \in
\mathcal{N}(E)$, and define the quadratic functional
\begin{equation} \label{eqn:least_squares_functional}
  J_E(v; u)
    = \sum_{E' \in \mathcal{N}(E)} \frac{\omega_{E,E'}}{2} \big\lvert u_{E'} - v(x_{E'})\big\rvert^2
  \quad \text{for $v \in P^1(\R^d)$}.
\end{equation}
We then compute the minimizing polynomial
\begin{equation*}
  v_E = \argmin_{v \in P^1(\R^d)} J_E(v; u)
    \quad \text{for all $E \in \grid$}
\end{equation*}
subject to the local mass-conservation property
\begin{math}
  v_E(x_E) = u_E.
\end{math}
It is easy to see that the minimization problem is well-posed, if
\begin{equation} \label{eqn:well_posedness_condition}
  \operatorname*{span} \{\omega_{E,E'}\,(x_{E'} - x_E) \mid E' \in \mathcal{N}(E) \} = \R^d.
\end{equation}
Next, we introduce a set of locally admissible linear functions (see, e.g.,
\cite{Hubbard1999}), given by
\begin{multline} \label{eqn:admissible_set}
  \mathcal{W}(E; u) = \bigl\{w \in P^1(\R^d) \mid w(x_E) = u_E \text{ and} \\
    \min\{u_E, u_{E'}\} \leq w(x_{E'}) \leq \max\{u_E, u_{E'} \}
      \text{ for all } E' \in \mathcal{N}(E) \bigr\}.
\end{multline}
Note that, by definition, the set $\mathcal{W}(E;u)$ is convex and non-empty,
since the constant function $w(x) = u_E$ is always admissible. A function
$v \in X^1_\grid$ is called admissible, if
\begin{equation*}
  v_{\restriction E} \in \mathcal{W}(E; \Pi^0_\grid v) \quad \text{for all $E \in \grid$}.
\end{equation*}
Having computed the linear polynomial $v_E$, an inexpensive projection onto the
set of admissible polynomials $\mathcal{W}(E; u)$ is given by a scaling of the
candidate gradient $\nabla v_E$. We define the mapping
\begin{math}
  \mathcal{R} : X^0_\grid \rightarrow X^1_\grid
\end{math}
by
\begin{equation*}
  \mathcal{R} u_{\restriction E}(x)
    = u_E + \alpha_E \nabla v_{\restriction E} \cdot (x - x_E)
      \quad \text{for all $E \in \grid$},
\end{equation*}
where $\alpha_E \in [0,1]$ is chosen maximal such that the image is admissible.
The scalar factor $\alpha_E$ can be computed explicitly.

\section{Constrained linear reconstruction}

In the previous section, limitation has been considered a separate step in the
definition of a stabilized reconstruction operator. Here, we follow a more
recent approach of recovering a suitably bounded approximate gradient in a
single step by means of local minimization problems. For example, \cite{May2013,
Chen2016} proposed to directly reconstruct an admissible solution through a
linear programming (LP) problem.

\begin{definition}[LP reconstruction]
Let $\grid$ be an arbitrary $d$-dimensional grid. The LP reconstruction
operator
\begin{math}
  \mathcal{R}: X^0_\grid \rightarrow X^1_\grid
\end{math}
is defined by
\begin{equation} \label{eqn:l1optimmod}
  \mathcal{R} u_{\restriction{E}} = w_E
    = \argmin_{w \in \mathcal{W}(E;u)}
      \sum_{E' \in \mathcal{N}(E)} \big\lvert u_{E'} - w(x_{E'})\big\rvert
\end{equation}
where the set of locally admissible functions $\mathcal{W}(E;u)$ is given by
Equation \eqref{eqn:admissible_set}.
\end{definition}

This reconstruction has been shown to be equivalent to the following LP problem:
\begin{equation*}
  \begin{aligned}
    & \text{maximize} \quad&
    & \sum_{E' \in \mathcal{N}( E )} \sign( u_{E'} - u_E ) (x_{E'} - x_E) \cdot \nabla w, \\
    & \text{subject to} \quad&
    & 0 \leq \sign(u_{E'} - u_E) \, (x_{E'} - x_E) \cdot \nabla w \leq \lvert u_{E'} - u_E \rvert.
  \end{aligned}
\end{equation*}
In \cite{May2013, Chen2016}, the authors propose to solve this problem by a
variant of the classical simplex algorithm.

The use of the $l^1$-Norm in the objective function in
Equation~\eqref{eqn:l1optimmod} might seem natural in the context of hyperbolic
conservation laws. For the numerical approximation of overdetermined problems,
the use of quadratic objective functions, e.g., least-squares
fits, are more common. The approach we propose may be summarized
as follows: for each grid element $E \in \grid$ we choose a
locally admissible polynomial $w_E \in \mathcal{W}(E;u)$ as the
best admissible fit in a least-squares sense to given piecewise
constant data.

\begin{definition}[QP reconstruction] \label{dfn:optimmod}
Let $\grid$ be an arbitrary $d$-dimensional grid. The QP reconstruction
operator
\begin{math}
  \mathcal{R}: X^0_\grid \rightarrow X^1_\grid
\end{math}
is defined by
\begin{equation} \label{eqn:optimmod}
  \mathcal{R} u_{\restriction{E}} = w_E = \argmin_{w \in \mathcal{W}(E;u)} J_E(w; u),
\end{equation}
where $J_E$ is defined as in Equation~\eqref{eqn:least_squares_functional} and
the set of locally admissible functions $\mathcal{W}(E;u)$ is given by Equation
\eqref{eqn:admissible_set}.
\end{definition}

First, observe that the optimization problem \eqref{eqn:optimmod} is equivalent
to a standard quadratic programming (QP) problem for the approximate gradient.
Indeed, using the notation
\begin{math}
  d_{E,E'} = x_{E'} - x_E
\end{math}
and
\begin{math}
  m_{E,E'} = u_{E'} - u_E
\end{math}
for $E' \in \mathcal{N}(E)$, a linear function $w_E$ is a solution to
\eqref{eqn:optimmod} if and only if $\nabla w_E$ solves the QP problem
\begin{equation} \label{eqn:optimmod_qp}
  \begin{aligned}
    & \text{minimize} \quad&
    & \frac{1}{2} \nabla w \cdot (H \nabla w) - g \cdot \nabla w, \\
    & \text{subject to} \quad&
    & 0 \leq \sign(m_{E,E'}) \, d_{E,E'} \cdot \nabla w \leq \lvert m_{E,E'} \rvert,
  \end{aligned}
\end{equation}
where $\sign(a) \in \{-1,1\}$ denotes the sign of $a \in \R$. The Hessian $H
\in \R^{d \times d}$ and the gradient $g \in \R^d$ of the objective function
are given by
\begin{align*}
  H & = \sum_{E' \in \mathcal{N}(E)} \, \omega_{E,E'} \, d_{E,E'} \otimes d_{E,E'}, \\
  g & = \sum_{E' \in \mathcal{N}(E)} \, \omega_{E,E'} \, m_{E,E'} \, d_{E,E'}.
\end{align*}
The matrix $H$ is positive definite due to assumption
\eqref{eqn:well_posedness_condition} on the choice of weights $\omega_{E,E'}$.
A QP problem can be solved efficiently in a small number of steps. For a
description of the numerical methods the reader is referred to standard
textbooks on constrained optimization.

With a different set of linear constraints, the reconstruction in
Definition~\ref{dfn:optimmod} was also proposed in \cite{Buffard2010}. However,
the authors restrict themselves to conforming triangular grids to compute the
exact solution to the arising QP. In contrast, we propose the use of an
active set strategy to solve the QP problem numerically. As a consequence,
we do not have to impose any restrictions on the grid dimension, the
conformity of the grid or the shape of individual grid elements.

The main result of this section is given in Theorem~\ref{thm:optimmod_minmod}.
In case of Cartesian meshes the QP reconstruction coincides with the well-known
and reliable Minmod limiter. In the following, let $\grid$ denote a
$d$-dimensional Cartesian grid of uniform grid width $h = (h_1, \ldots, h_d)$.
For fixed $E \in \grid$, we will denote its neighbors by $E_i^\pm$, $i = 1,
\ldots, d$, defined by
\begin{equation*}
  E_i^\pm = E \pm h_i\,e_i.
\end{equation*}

\begin{theorem} \label{thm:optimmod_minmod}
Let $\grid$ be a $d$-dimensional Cartesian grid of uniform grid width $h =
(h_1, \ldots, h_d)$, and let $E \in \grid$ be a fixed element. Then, the exact
solution $\nabla w_E$ to the QP problem \eqref{eqn:optimmod_qp} is given by
\begin{equation*}
  \nabla w_E = \sum_{i = 1}^d \minmod\left(\frac{u_{E_i^+} - u_E}{h_i},
    \frac{u_E - u_{E_i^-}}{h_i}\right) \, e_i.
\end{equation*}
In particular, $\nabla w_E$ is independent of the choice of weights
$\omega_{E,E_i^\pm}$.
\end{theorem}

\begin{proof}
For a Cartesian grid, we have $d_{E,E_i^\pm} = \pm h_i\,e_i$ and the Hessian
$H$ becomes a diagonal matrix with
\begin{equation*}
  H_{ii} = h_i^2\,(\omega_i^+ + \omega_i^-),
\end{equation*}
where we denoted $\omega_i^\pm = \omega_{E,E_i^\pm}$. Similarly, using the
notations $m_i^\pm = m_{E,E_i^\pm}$, the inequality constraints simplify to box
constraints for $\partial_i w$. Therefore, the optimization problem is
equivalent to the $d$ one-dimensional quadratic problems
\begin{equation*}
  \begin{aligned}
    & \text{minimize} \quad&
    & J_i(\partial_i w) = \frac{\omega_i^+ + \omega_i^-}{2}\,(h_i\,\partial_i w)^2 - (\omega_i^+\,s_i^+\,\abs{m_i^+} + \omega_i^-\,s_i^-\,\abs{m_i^-})\,h_i\,\partial_i w \\
    & \text{subject to} \quad&
    & 0 \leq s_i^\pm\,\partial_i w \leq \frac{1}{h_i}\,\lvert m_i^\pm \rvert,
  \end{aligned}
\end{equation*}
where we denote $s_i^\pm = \pm \sign(m_i^\pm)$. Now, if $s_i^+ \neq s_i^-$ or
either of $m_i^\pm$ vanishes, the constraints require $\partial_i w = 0$, which
agrees with the Minmod limiter.

Otherwise, let $s_i = s_i^+ = s_i^-$. The global minimum of each functional
$J_i$ is attained for
\begin{equation*}
  s_i\,\partial_i w
    = \frac{1}{h_i} \frac{\omega_i^+\,\abs{m_i^+} + \omega_i^-\,\abs{m_i^-}}{\omega_i^+ + \omega_i^-}
    \ge \frac{1}{h_i} \min\bigl\{\abs{m_i^+}, \abs{m_i^-} \bigr\}.
\end{equation*}
The opposite inequality follows directly from the constraints and we conclude
\begin{equation*}
  \partial_i w
    = \frac{s_i}{h_i} \min\bigl\{\abs{m_i^+}, \abs{m_i^-} \bigr\}
    = \frac{1}{h_i} \minmod(m_i^+, m_i^-),
\end{equation*}
which proves the statement.
\qed
\end{proof}

\section{Numerical results}

In this section we want to study the accuracy and efficiency of the QP
reconstruction, Definition~\ref{dfn:optimmod}. To this end, generic
implementations of all reconstruction operators discussed in this paper were
written within the \textsc{Dune} framework \cite{Bastian2008, Blatt2016}. For
the computations in Section~\ref{sct:euler_results}, the parallel grid library
\textsc{Dune-ALUGrid} \cite{Alkaemper2016} was used.

\subsection{Nonlinear problem admitting a smooth solution}

The first benchmark problem is taken from \cite[Chapter 3.5]{Kroener1997}. We
consider in the unit square $\Omega = (0,1)^2$ a nonlinear balance law
\begin{equation*}
  \partial_t u + \partial_1 u^2 + \partial_2 u^2 = s
    \quad \text{in $\Omega \times (0,\tfrac{3}{10})$}.
\end{equation*}
We want to numerically recover a prescribed smooth solution given by
\begin{equation*}
  u(x,t) = \frac{1}{5} \sin\left(2\pi(x_1 - t)\right)\sin\left(2\pi(x_2 - t)\right).
\end{equation*}
Note that the right hand side of Equation~\eqref{eqn:muscl_scheme} must be
extended by the discrete source term. The Initial data, source term and
Dirichlet boundary values are then determined from the prescribed solution.

\begin{figure}[b]
  \centering
  \includegraphics[width=0.33\linewidth]{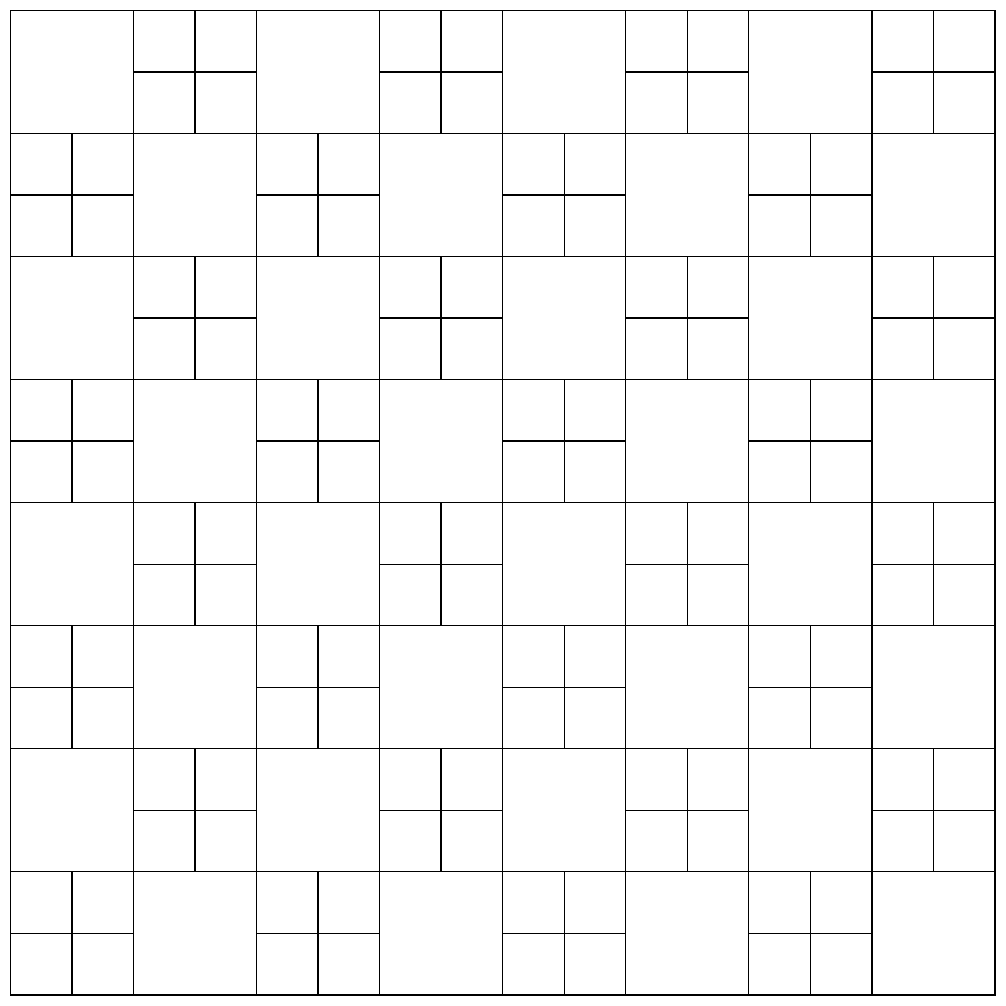} \quad
  \includegraphics[width=0.33\linewidth]{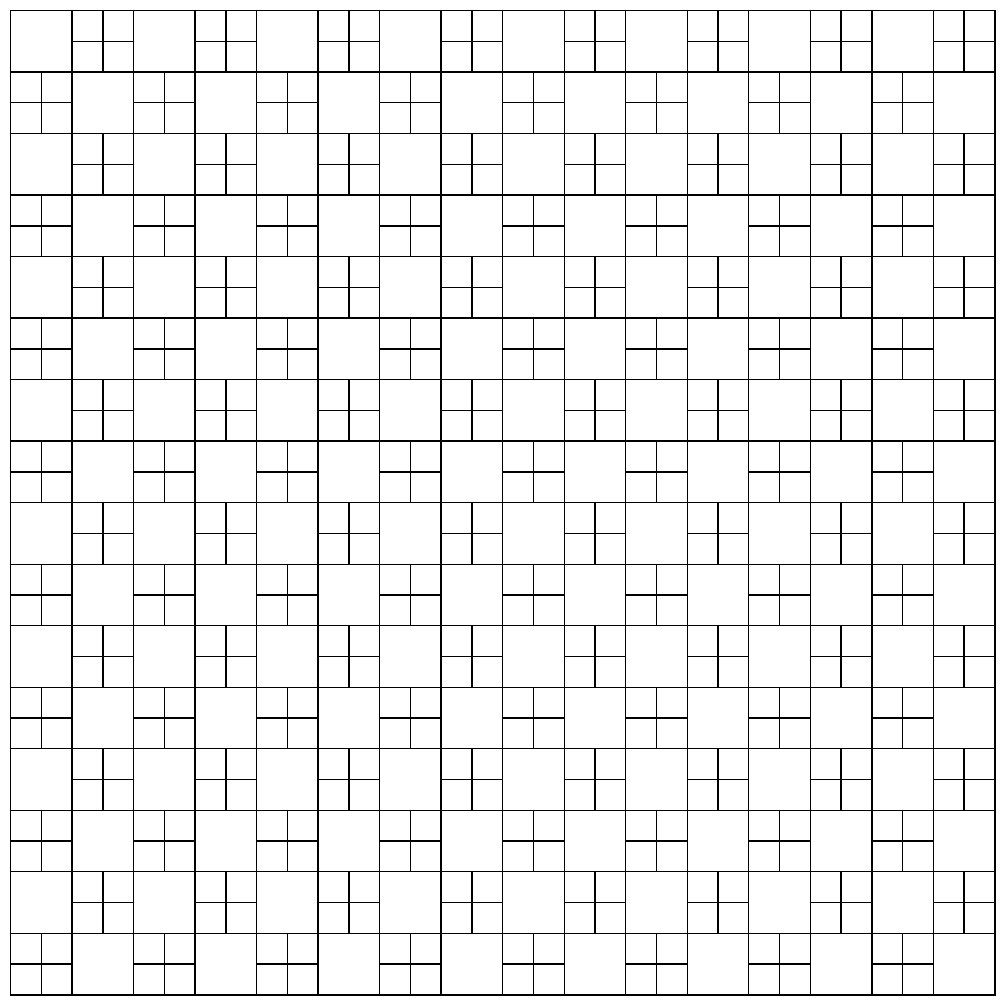}
  \caption{
    Non-conforming quadrilateral meshes with $160$ and $640$ elements, generated
    by refining every second element of a Cartesian grid.
 }
  \label{fig:rotation_checkerboard}
\end{figure}

We consider two different types of domain discretizations: a series of
conforming triangular grids generated through refinement of a coarse Delaunay
triangulation with \num{123} elements and a series of non-conforming
quadrilateral grids resulting from a checkerboard-like refinement rule,
Figure~\ref{fig:rotation_checkerboard}. This latter is
particularly challenging as the number of non-conformities grows
linearly in the number of elements.

\begin{table}[t]
  \centering
  \caption{Nonlinear benchmark problem admitting a smooth solution,
    $L^1$-errors and convergence rates at final time $t = \tfrac{3}{10}$.}
  \label{tbl:nonlinear_benchmark}

  \begin{subtable}{\linewidth}
  \centering
  \caption{Conforming unstructured triangular meshes.}

  \begin{tabular}{rcS[table-format=1.2e-2]S[table-format=1.2]cS[table-format=1.2e-2]S[table-format=1.2]cS[table-format=1.2e-2]S[table-format=1.2]}
  \toprule
    & & \multicolumn{2}{c}{Limited LSF}& & \multicolumn{2}{c}{QP Reconstruction}& & \multicolumn{2}{c}{LP Reconstruction} \\ \cline{3-4} \cline{6-7} \cline{9-10}\noalign{\smallskip}
    {Elements}& & {$\|u - u_\mathcal{G}\|_{L^1}$} & {EOC}& & {$\|u - u_\mathcal{G}\|_{L^1}$} & {EOC}& & {$\|u - u_\mathcal{G}\|_{L^1}$} & {EOC} \\
  \midrule
    \num{123} & & 1.35e-02 & {---} & & 9.86e-03 & {---} & & 1.08e-02 & {---} \\
    \num{492} & & 5.59e-03 & 1.27 & & 3.09e-03 & 1.68 & & 3.75e-03 & 1.52 \\
    \num{1968} & & 2.24e-03 & 1.32 & & 9.09e-04 & 1.76 & & 1.15e-03 & 1.71 \\
    \num{7872} & & 1.05e-03 & 1.10 & & 2.61e-04 & 1.80 & & 3.58e-04 & 1.68 \\
    \num{31488} & & 5.36e-04 & 0.97 & & 7.31e-05 & 1.84 & & 1.07e-04 & 1.74 \\
  \bottomrule
  \end{tabular}
  \end{subtable} \\[\medskipamount]
  \begin{subtable}{\linewidth}
  \centering
  \caption{Non-conforming quadrilateral meshes.}
  \begin{tabular}{rcS[table-format=1.2e-2]S[table-format=1.2]cS[table-format=1.2e-2]S[table-format=1.2]cS[table-format=1.2e-2]S[table-format=1.2]}
  \toprule
    & & \multicolumn{2}{c}{Limited LSF}& & \multicolumn{2}{c}{QP Reconstruction}& & \multicolumn{2}{c}{LP Reconstruction} \\ \cline{3-4} \cline{6-7} \cline{9-10}\noalign{\smallskip}
    {Elements}& & {$\|u - u_\mathcal{G}\|_{L^1}$} & {EOC}& & {$\|u - u_\mathcal{G}\|_{L^1}$} & {EOC}& & {$\|u - u_\mathcal{G}\|_{L^1}$} & {EOC} \\
  \midrule
    \num{160} & & 1.87e-02 & {---} & & 1.22e-02 & {---} & & 1.24e-02 & {---} \\
    \num{640} & & 8.22e-03 & 1.19 & & 4.44e-03 & 1.46 & & 4.44e-03 & 1.48 \\
    \num{2560} & & 3.18e-03 & 1.37 & & 1.30e-03 & 1.78 & & 1.30e-03 & 1.77 \\
    \num{10240} & & 1.45e-03 & 1.13 & & 4.17e-04 & 1.63 & & 4.21e-04 & 1.63 \\
    \num{40960} & & 7.27e-04 & 1.00 & & 1.49e-04 & 1.48 & & 1.51e-04 & 1.48 \\
  \bottomrule
  \end{tabular}
  \end{subtable}
\end{table}

Table~\ref{tbl:nonlinear_benchmark} shows the $L^1$-errors and convergence
rates for a second-order Lax-Friedrichs scheme using the three different
reconstruction operators. Observe that in all cases the simple limited least
squares fit results in a first-order approximation only. The QP and LP
reconstructions give much better results with the QP reconstruction having a
slight edge over the LP reconstruction in case of triangular meshes; in case of
the highly non-conforming quadrilateral meshes, however, the approximation
order of both reconstructions drops to around $1.5$.

\subsection{Linear problem}

In this section we are interested in the efficiency of the QP and LP
reconstructions. To this end, we consider the well-known \emph{solid body
rotation} benchmark problem proposed by \citet{LeVeque1996}. In the unit
square $\Omega = (0,1)^2$ we consider the counterclockwise rotation about the
center with periodicity $T = 2\pi$,
\begin{align*}
  \partial_t u + \nabla \cdot(\bm{v} u) & = 0 \quad \text{in} \Omega \times (0,T), \\
  \bm{v}(x,y) & = \bigl(\tfrac{1}{2} - y, x - \tfrac{1}{2}\bigr).
\end{align*}
The initial data consists of a slotted cylinder, a cone and a smooth hump, each
of which is restricted to a circular domain of radius $r = 0.15$,
\begin{equation*}
  u_0(x) = \begin{cases}
      1 - \chi_{[0.475, 0.525]}(x_1)\,\chi_{[0,0.85]}(x_2)
        & \text{if } \abs{x - x_s} \leq r, \\
      1 - \frac{\abs{x - x_c}}{r}
        & \text{if } \abs{x - x_c} \leq r, \\
      \frac{1}{4} + \frac{1}{4}\cos\left(\pi \frac{\abs{x - x_h}}{r}\right)
        & \text{if } \abs{x - x_h} \leq r, \\
      0 & \text{otherwise}.
    \end{cases}
\end{equation*}
where $x_s = (0.5, 0.75)$, $x_c = (0.5, 0.25)$ and $x_h = (0.25, 0.5)$, respectively.

Figure~\ref{fig:benchmark_runtime_vs_error} shows two plots of the $L^1$-errors
over the computation time. We used the exact same series of triangular and
quadrilateral meshes as in the previous section. In both cases the LP and QP
reconstructions perform much better than the limited least squares fit. While
the latter is easy to implement and inexpensive to compute, solving a quadratic
or linear minimization problem on each cell results in a more efficient scheme
even in the presence of strong discontinuities.

\begin{figure}[t]
  \centering
  \includegraphics[width=0.475\linewidth]{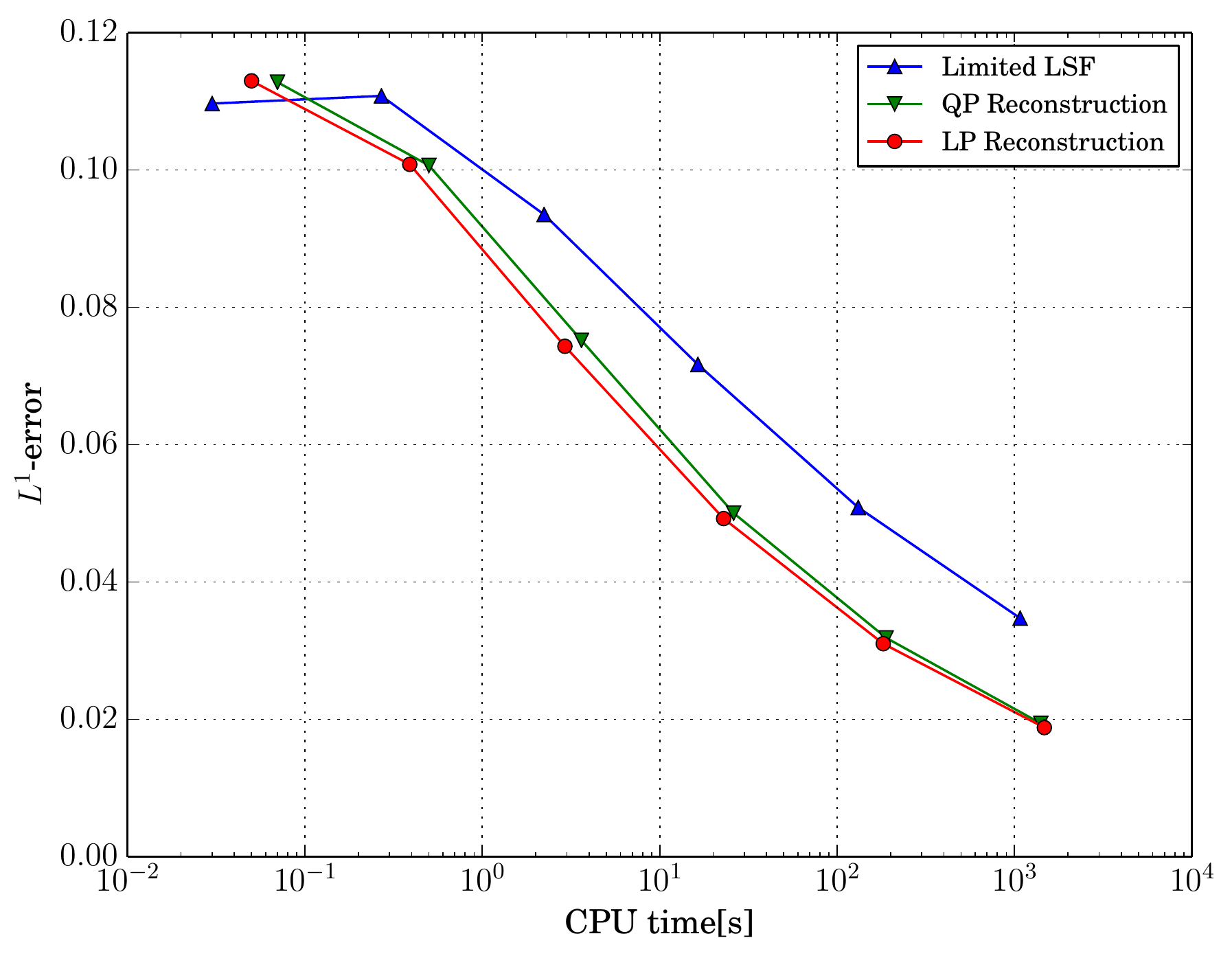}
  \includegraphics[width=0.475\linewidth]{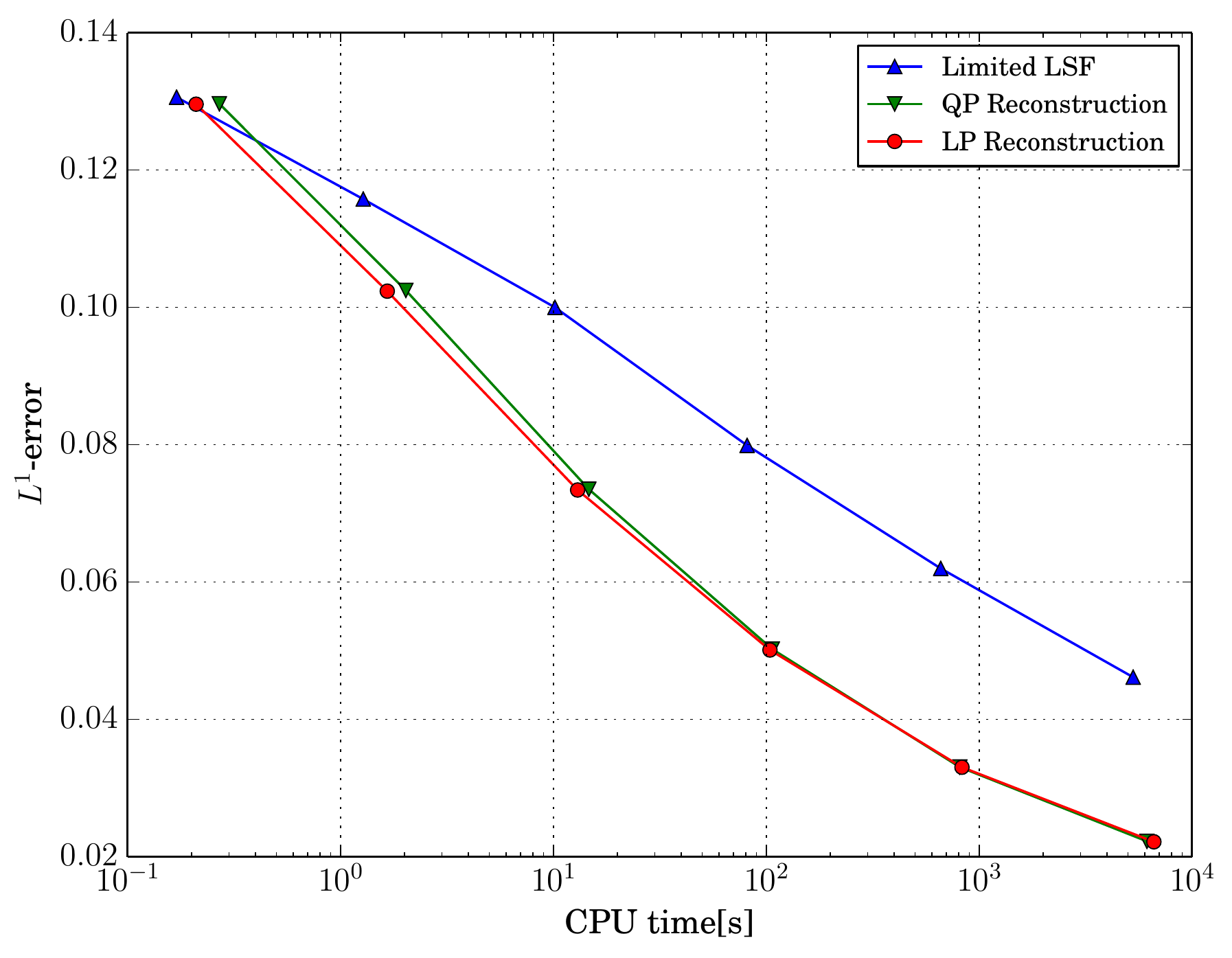}
  \caption{Efficiency of different reconstructions on two
      sequences of meshes of varying resolution for the solid body rotation
      benchmark problem.}
  \label{fig:benchmark_runtime_vs_error}
\end{figure}

\subsection{Euler equations of gas dynamics}
\label{sct:euler_results}

In this final section we consider the $d$-dimensional Euler equations of gas
dynamics,
\begin{equation*}
  \partial_t U + \nabla \cdot F(U) = 0,
\end{equation*}
where
\begin{math}
  U = (\rho, \rho v, E) \in \R^{d+2}
\end{math}
is the vector of conserved quantities, i.e., the density $\rho$, the momentum
$\rho v$, and the total energy density $E$. By $v \in \R^d$ we denote the
primitive particle velocity. The convective flux is given by
\begin{equation*}
  F(U) = \begin{pmatrix}
           \rho v \\
           \rho v \otimes v + p \, \mathbb{I}_d \\
           (E + p) v
         \end{pmatrix} \in \R^{(d + 2) \times d},
\end{equation*}
where $\mathbb{I}_d$ denotes the $d$-dimensional identity matrix.
The pressure $p = p(U)$ is given by the equation of state
\begin{math}
  p(U) = (\gamma - 1)\left(E - \tfrac{\rho}{2} \abs{v}^2 \right)
\end{math}
with the adiabatic constant $\gamma = \num{1.4}$. A state $U$ is considered
physical if the density $\rho$ and the pressure $p$ are strictly positive.
The set of states is thus given by
\begin{equation*}
  \mathcal{U} = \bigl\{U \in \R^{d+2} \mid \rho > 0 \text{ and } p > 0 \bigr\}.
\end{equation*}

It is well-known that unphysical values in an approximate solution typically
lead to the immediate break-down of a numerical simulation. Therefore, we
use a further stabilized modification of the QP reconstruction,
Definition~\ref{dfn:optimmod}, to the stabilization of the finite volume
scheme \eqref{eqn:muscl_scheme} applied to the vector of primitive
variables $W = (\rho, v, E)$. For all $E \in \grid$ we consider a modified
set of locally admissible linear functions
\begin{multline*}
  \mathcal{W}'(E;u)
    = \bigl\{W \in [P^1(\R^d)]^{d+2} \mid W(x_E) = W_E \text{ and } \\
    \min\{W_E, W_{E,E'}\} \leq W(x_{E,E'}) \leq \max\{W_E, W_{E,E'} \}
      \text{ for all } E' \in \mathcal{N}(E) \bigr\},
\end{multline*}
such that all admissible polynomials are physical at least in the midpoints of
inter-element intersections. To this end we define the intermediate values
$W_{E,E'}$, $E' \in \mathcal{N}(E)$, $e = E \cap E'$,
\begin{equation*}
  W_{E,E'}
    = \frac{\abs{x_{E'} - x_{e}}}{\abs{x_E - x_{e}} + \abs{x_{E'} - x_{e}}}\,W_E
      + \frac{\abs{x_E - x_{e}}}{\abs{x_E - x_{e}} + \abs{x_{E'} - x_{e}}} W_{E'}.
\end{equation*}
We define the positivity-preserving reconstruction operator
\begin{math}
  \mathcal{R}': X^0_\grid \rightarrow X^1_\grid
\end{math}
such that for each element $E \in \grid$ the reconstruction is the best
positivity-preserving $w_E \in \mathcal{W}'(E;u)$ that fits the piecewise
constant data in a least-squares sense.

We are interested in the solution of the three-dimensional shock tube
experiment. The computational domain is a cylinder
\begin{math}
  \Omega = \bigl\{x \in \R^3 \mid -1 < x_1 < 1 \text{ and } \sqrt{x_2^2 + x_3^2} < 0.2 \bigr\}.
\end{math}
In primitive variables, the initial data is given by
\begin{equation*}
  W(\cdot, 0) = \begin{cases}
      W_L & \text{if $x_1 < 0$}, \\
      W_R & \text{otherwise}.
    \end{cases}
\end{equation*}
The left and right states $W_L, W_R$ for two different test settings are chosen
as in \cite[Chapter 4.3]{Toro1997}, see Table~\ref{tbl:riemann_data}. The
first test problem is known as Sod test problem; its solution consists of a
left rarefaction, a contact discontinuity and a right shock. The second test
problem is the so-called \emph{123} problem; in this case the solution
consists of two rarefactions and a trivial stationary contact discontinuity.
The latter benchmark test is particularly challenging as small densities and
pressures occur. We solve the Riemann problems for times $t \leq 0.5$ in case
of the Sod problem, and for times $t \leq 0.15$ in case of the \emph{p123}
problem. At the left and right boundary, we prescribe $W = W_L$ and $W = W_R$,
respectively. Otherwise, slip boundary conditions are imposed. The solutions
can be computed in a quasi-exact manner as in case of the one-dimensional
Euler equations, see, e.g., \cite{Toro1997}.

\begin{table}[t]
  \centering
  \caption{Initial left and right states in primitive variables for the
    three-dimensional shock tube experiments.}
  \label{tbl:riemann_data}

  \begin{tabular}{p{2cm}S[table-format=1.0]S[table-format=-1.0]S[table-format=1.0]S[table-format=1.0]S[table-format=1.1]cS[table-format=1.3]S[table-format=1.0]S[table-format=1.0]S[table-format=1.0]S[table-format=1.1]}
  \toprule
     & \multicolumn{5}{c}{$W_L$} & & \multicolumn{5}{c}{$W_R$} \\ \cline{2-6} \cline{8-12} \noalign{\smallskip}
     & {$\rho$} & {$v_1$} & {$v_2$} & {$v_3$} & {$p$} & & {$\rho$} & {$v_1$} & {$v_2$} & {$v_3$} & {$p$} \\
  \midrule
    Sod problem & 1 & 0 & 0 & 0 & 1 & & 0.125 & 0 & 0 & 0 & 0.1 \\
    \emph{p123} problem & 1 & -2 & 0 & 0 & 0.4 & & 1 & 2 & 0 & 0 & 0.4 \\
  \bottomrule
  \end{tabular}
\end{table}

\begin{figure}[b]
  \centering
  \includegraphics[width=0.5\linewidth]{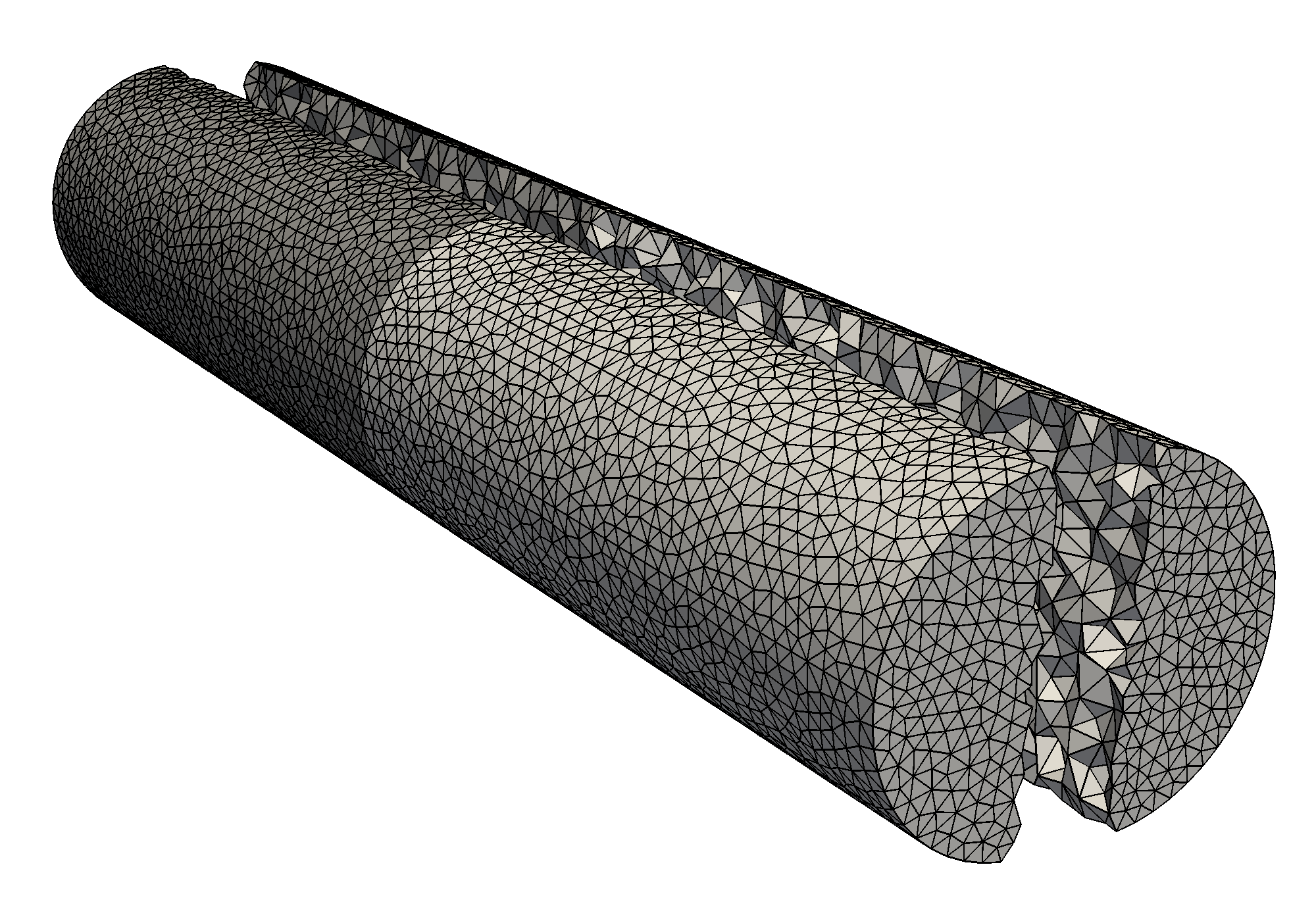}
  \caption{Unstructured tetrahedral mesh for the discretization of the
    three-dimensional shock tube.}
  \label{fig:shock_tube}
\end{figure}

The computational domain is discretized by a series of unstructured, affine
tetrahedral grids. Instead of refining the coarsest grid several times, each
grid was created separately using the Gmsh mesh generator \cite{Geuzaine2009}.
It was ensured that the discontinuity in the initial data is exactly resolved
by the grid. An example grid is shown in Figure~\ref{fig:shock_tube}. For the
MUSCL-type scheme, we use a Harten-Lax-van Leer (HLL) numerical flux. For each
grid we computed the $L^1$-errors against the quasi-exact solution in primitive
variables; the results are shown in Table~\ref{tbl:shock_tube}. The
experimental order of convergence of roughly $0.85$ for the Sod problem and
$0.8$ for the \emph{p123}-problem is significantly higher than the $0.5$
expected from a first-order scheme.

\begin{table}[t]
  \centering
  \caption{Three-dimensional shock tube experiment, $L^1$-errors and
    convergence rates against the quasi-exact solutions to the Riemann problems
    under consideration.}
  \label{tbl:shock_tube}

  \begin{tabular}{rcS[table-format=1.2e-2]cS[table-format=1.2e-2]S[table-format=1.2]cS[table-format=1.2e-2]S[table-format=1.2]}
  \toprule
    & & & & \multicolumn{2}{c}{Sod problem}& & \multicolumn{2}{c}{\emph{p123} problem} \\ \cline{5-6} \cline{8-9}\noalign{\smallskip}
    {Elements} & & {$h_\grid$} & & {$\|W - W_\grid\|_{L^1}$} & {EOC} & & {$\|W - W_\grid\|_{L^1}$} & {EOC} \\
  \midrule
    \num{288} & & 3.48e-01 & & 4.05e-02 & {---} & & 6.21e-02 & {---} \\
    \num{1680} & & 2.13e-01 & & 2.57e-02 & 0.93 & & 4.39e-02 & 0.71 \\
    \num{9937} & & 1.11e-01 & & 1.46e-02 & 0.86 & & 2.54e-02 & 0.83 \\
    \num{71788} & & 5.66e-02 & & 8.49e-03 & 0.81 & & 1.48e-02 & 0.81 \\
    \num{559593} & & 2.88e-02 & & 4.73e-03 & 0.87 & & 8.55e-03 & 0.82 \\
    \num{4396447} & & 1.50e-02 & & 2.65e-03 & 0.89 & & 5.28e-03 & 0.74 \\
  \bottomrule
  \end{tabular}
\end{table}

\section{Conclusion}

In this paper we proposed a new QP reconstruction method for MUSCL-type finite
volume schemes. For each grid element we computed the best admissible fit in a
least-squares sense. We showed that the QP reconstruction generalizes the
multidimensional Minmod reconstruction for Cartesian meshes. No restrictions to
the grid dimension, the conformity of the grid or the shape of individual grid
elements were made. The local cell problems only involve data associated with
direct neighbors of a grid element and thus preserve the locality of the
numerical method. We compared our reconstruction against similar techniques
proposed in the literature. By numerical experiments, we showed that the
minimization problems are indeed inexpensive to solve and yield a more accurate
and efficient approximation.

\bibliographystyle{abbrvnat}
\bibliography{preprint}

\end{document}